\documentclass{cpc}

\usepackage{amssymb}

\newtheorem{lem}[theorem]{Lemma}
\newtheorem{cor}[theorem]{Corollary}
\newtheorem{prop}[theorem]{Proposition}
\newtheorem{defn}[theorem]{Definition}

\title{Simple containers for simple hypergraphs}

\author[David Saxton and Andrew Thomason]%
{D\ls A\ls V\ls I\ls D\ns S\ls A\ls X\ls T\ls O\ls N\thanks{Supported in
    part by the Engineering and Physical Sciences Research Council, and
    CNPq bolsa PDJ.}
  \ns and\ns
  A\ls N\ls  D\ls R\ls E\ls W\ns T\ls H\ls O\ls M\ls A\ls S\ls O\ls N\\
  \affilskip IMPA, Estrada Dona Castorina 110, Rio de Janeiro, Brasil
  22460-320 \\
  \affilskip DPMMS, Centre for Mathematical Sciences, Wilberforce
  Road, Cambridge, UK\\
  {\tt saxton@impa.br} and {\tt A.G.Thomason@dpmms.cam.ac.uk}\\
}

\begin{document}
\maketitle
\begin{abstract}
  We give an easy method for constructing containers for simple
  hypergraphs. Some applications are given; in particular, a very
  transparent calculation is offered for the number of $H$-free
  hypergraphs, where $H$ is some fixed uniform hypergraph.
\end{abstract}

\section{Introduction}

The notion of a collection of containers for a hypergraph was introduced by
the authors in~\cite{ST1}. A collection of containers for a hypergraph
$G$ is a collection $\cal C$ of subsets of $V(G)$ such that every
independent set $I$ is a subset of some member $C\in{\cal C}$. (A subset of
$V(G)$ is independent if it contains no edge.)

The notion was developed further in~\cite{ST2} and several applications
given; related methods and results were proved by Balogh, Morris and
Samotij~\cite{BMS}. These results have since been applied by other
authors.

Our purpose here is to revisit the method of~\cite{ST1}, and to combine it
with a twist that makes it much more widely applicable. It is true that the
method of~\cite{ST2} is not too complicated, and the consequences are often
best possible, but it is subtle. The method of~\cite{ST1}, on the
other hand, is not optimal; nevertheless it is very simple, and it is
particularly transparent. It is sufficient, for example, for counting the
number of $H$-free hypergraphs (see Corollary~\ref{cor:ffree_count}), and
hence it offers a very elementary and straightforward proof of this result.

The method of~\cite{ST1} applies to {\em simple} or {\em linear}
hypergraphs, that is, hypergraphs in which no two edges share more than
one vertex. The container theorem there was as follows. We use the term
{\em$r$-graph} to mean an $r$-uniform hypergraph, where $r\ge2$ always.

\begin{prop}[\cite{ST1}]\label{prop:curly_c}
  Let $G$ be a $d$-regular simple $r$-graph.  If
  $d$ is large,  there is a collection of sets $\mathcal{C}$ of subsets
  of $V(G)$ satisfying
 \begin{itemize}
 \item if $I \subset V(G)$ is independent, there is some $C \in
   \mathcal{C}$ with $I \subset C$,
 \item $|C| \leq (1-1/4r^2)|G|$ for every $C \in \mathcal{C}$,
 \item $|\mathcal{C}| \leq 2^{\alpha|G|}$ where $\alpha=(1/d)^{1/(2r-1)}$.
 \end{itemize}
 \end{prop}

This proposition is not quite as stated in~\cite{ST1}, but it is pretty much
explicit in the proof of Theorem~1.1 that follows Theorem~3.1.

Two drawbacks limit the applicability of
Proposition~\ref{prop:curly_c}. The first is that many popular container
applications require containers with $e(G[C])$ small, rather than $|C|$
small. The second is that it applies only to regular $r$-graphs. Container
results are more useful when they can be applied iteratively. That is,
given an independent set $I$ in a hypergraph~$G$, we can apply the proposition
once to obtain a container $C$ for~$I$, but we would then like to apply the
proposition again, this time to the hypergraph $G[C]$ (of which $I$ is
still an independent set), thus finding a smaller container $C'\subset
C$. If possible we would then repeat this procedure until very small
containers are obtained. The snag with this procedure as it stands is that
it is unlikely that $G[C]$ is regular even if $G$ is, and so iteration is
not possible. Of course, Proposition~\ref{prop:curly_c} still applies to
graphs that are ``somewhat'' regular (indeed, this follows directly from
Theorem~3.1 of \cite{ST1}), but not in a particularly strong way.

Both these drawbacks can be overcome by adapting the proof of Theorem~3.1
in~\cite{ST1} to use the notion of {\em degree measure}, which we discuss
in~\S\ref{sec:degree}. This yields a version of
Proposition~\ref{prop:curly_c} in which $C$ is bounded in degree measure,
namely Theorem~\ref{thm:simplecont}. Iterated applications of this theorem
give the following result, much stronger and more useful than
Proposition~\ref{prop:curly_c}, and the main theorem of the present paper.

\begin{theorem}\label{thm:sparsecont}
  Let $G$ be a simple $r$-graph of average degree $d$. Let
  $0<\delta<1$. If $d$ is large enough, then there is a collection of sets
  $\mathcal{C}$ of subsets of $V(G)$ satisfying
 \begin{itemize}
 \item if $I \subset V(G)$ is independent, there is some $C \in
   \mathcal{C}$ with $I \subset C$,
 \item $e(G[C])<\delta e(G)$ for every $C \in \mathcal{C}$,
 \item $|\mathcal{C}| \leq 2^{\beta |G|}$ where
   $\beta=(1/d)^{1/(2r-1)}$.
 \end{itemize}
\end{theorem}

Observe that Theorem~\ref{thm:sparsecont} differs from
Proposition~\ref{prop:curly_c} only in that the condition of $d$-regularity
is replaced by that of average degree~$d$, and the conclusion giving a
bound on $|C|$ is replaced by a bound on $e(G[C])$. This bound on $e(G[C])$
implies, for regular~$G$, that $|C| \le (1 - 1/r + \delta/r)n$ (see the
discussion in~\S\ref{sec:degree}), which is essentially best possible, but
is nonetheless a weaker condition than the bound on the number of edges.
Thus Theorem~\ref{thm:sparsecont} is a
generalization of Proposition~\ref{prop:curly_c}.

\section{Degree measure}\label{sec:degree}

The notion of degree measure was introduced in~\cite{ST2}.
In the definition below, $d(v)$ is the degree of the
vertex~$v$.

\begin{defn}\label{def:mu}
Let $G$ be an $r$-graph of order~$n$ and average degree~$d$. Let $S\subset
V(G)$. The {\em degree measure} $\mu(S)$ of $S$ is defined by
$$
\mu(S)\,=\,\frac{1}{nd}\,\sum_{v\in S}d(v)\,.
$$
\end{defn}

We note some immediate properties of degree measure. First, for any $S\subset
V(G)$,
\begin{equation}\label{eqn:musmall}
e(G[S])\,\le\, {1\over r}\sum_{v\in S} d(v)\, =\, {d|G|\over
  r}\mu(S)\,=\,\mu(S)e(G)\,,
\end{equation}
so $\mu(S)$ small implies $e(G[S])$ small.

Moreover, sets of large measure must contain many edges. Indeed, writing
$\overline S$ for $V(G)-S$ and $e(\overline S,S)$ for the number of edges
meeting both $\overline S$ and $S$, we have
$$
(r-1)nd\mu(\overline S) = (r-1)\sum_{v\notin S}d(v)\,
\ge\, (r-1) e(\overline S, S) \,\ge\,\left\{\sum_{v\in
    S}d(v)-re(G[S])\right\} \,,
$$
that is, $(r-1)\mu(\overline S)\ge \mu(S)-re(G[S])/nd$. Since
$\mu(\overline S)=1-\mu(S)$ this means
\begin{equation}\label{eqn:mubig}
e(G[S])\,\ge\,(\mu(S)-1+{1\over r})\,nd\,.
\end{equation}
In particular, the measure of an independent set cannot exceed
$1-1/r$. Furthermore, if $G$ is regular, then degree and uniform measures
coincide; in this case, the inequality $e(G[S])\le \delta e(G)$ together 
with~(\ref{eqn:mubig}) implies $|S|\le(1-1/r+\delta/ r)n$, as mentioned in
the introduction.

We can now state the theorem which is at the heart of the present paper.
This theorem is already sufficiently powerful for obtaining non-trivial
results, such as in list colouring.

\begin{theorem}\label{thm:simplecont}
  Let $G$ be a simple $r$-graph of average degree
  $d$.  If $d$ is large, there is a collection of sets $\mathcal{C}$
  of subsets of $V(G)$ satisfying
 \begin{itemize}
 \item if $I \subset V(G)$ is independent, there is some $C \in
   \mathcal{C}$ with $I \subset C$,
 \item $\mu(C)\le 1-1/4r^2$ for every $C \in \mathcal{C}$,
 \item $|\mathcal{C}| \leq 2^{\alpha |G|}$ where $\alpha=(1/d)^{1/(2r-1)}$.
 \end{itemize}
\end{theorem}

The proof of Theorem~\ref{thm:simplecont} follows quite closely the proof
of Theorem~3.1 in~\cite{ST1}, but modifications are needed to
accommodate the presence of both uniform and degree measures. However,
nothing stronger than Markov's inequality is needed.

The spirit of the proof is readily explained. We need to identify a set of
vertices that are not in~$I$; then $C$ will be the remaining vertices. We
shall show that there are three small subsets $R$, $S$ and $T$ of $V=V(G)$,
such that $R$, $S$ and $T$ determine such a set $V \setminus C$ disjoint
from~$I$. This means that the number of different container sets $C$ that
are so specified is at most the number of triples of small subsets
$(R,S,T)$; this number is not large and this is where the bound on $|{\cal
  C}|$ comes from.

How, then, can we specify $R$, $S$ and $T$ in such a way as to enable us to
identify a set $V \setminus C$ of vertices not in~$I$? There are no edges with all
$r$ vertices inside~$I$, but there are many edges altogether. So there must
be a number $j$, $0\le j<r$, such that there are significantly fewer edges
with $j+1$ vertices in $I$ than there are edges with $j$ vertices in~$I$.
We might then expect to find a substantial set $D\subset V \setminus I$ of vertices
each lying in many of the latter kind of edges. So we pick small subsets
$R\subset I$ and $S\subset V \setminus I$ at random, and look at the set
\[
 \Gamma_j(R,S) = \{ v \in V : \mbox{ there is an edge } \{v\} \cup f \cup g
\, \mbox{ with } f \in R^{(j)} \mbox{ and } g \in S^{(r-j-1)}\}\,,
\]
where $R^{(j)} = \{Y \subset R : |Y|=j\}$, etc.
Notice that $\Gamma_j(R,S)$ is determined by $R$ and $S$. If we write
$T=\Gamma_j(R,S)\cap I$ then clearly $C=(V\setminus \Gamma_j(R,S))\cup T$
is a container for $I$ that is specified by $(R,S,T)$. Now $R$ and $S$ are
small by definition, and we expect $T$ also to be small, because there are
few edges with $j+1$ vertices in~$I$. On the other hand, vertices of $D$
have a good chance of lying inside $\Gamma_j(R,S)$, so we expect
$\Gamma_j(R,S)$ to contain much of $D$ and so have substantial measure,
meaning that $\mu(C)$ is bounded away from one. This is the heart of the
proof.

\begin{proof}[Proof of Theorem~\ref{thm:simplecont}.]
Let $V = V(G)$ be the vertex set of $G$ of size $n=|V|$ and $E = E(G)$ the edge set. 
For sets $R, S \subset V$ and $0 \leq j \leq r-1$, let $\Gamma_j(R,S)$ be
as defined above. Given subsets $R, S, T \subset V$, let
\[
 C_j(R,S,T) = \left\{
   \begin{array}{c@{\qquad}l}
     V \setminus (\Gamma_j(R,S) \setminus T) & \mbox{if
       $\mu(\Gamma_j(R,S)\setminus T)\ge 1/4r^2$}\\
     \emptyset &\mbox{otherwise.}
   \end{array}
   \right.
\]
Note that $\mu(C_j(R,S,T))\le 1-1/4r^2$ by definition.
We will show that for every independent set $I$, there are small subsets
$R,S,T \subset V$ such that $I \subset C_j(R,S,T)$. Specifically, let
$$
u = \frac{1}{\sqrt{3r}}\left( 6r\over d\right)^{1/2(r-1)}   
\mbox{\quad and \quad} q=15ru\,.
$$
Note that $q$ is small if $d$ is large (depending on $r$). We now define the
collection $\cal C$ by ${\cal C}=\{C_j(R,S,T): 0\le j\le r-1,\ |R|, |S|,
|T| \le qn\}$. Then
\[
 |\mathcal{C}| \,\leq \,r (qn)^3{n \choose qn}^3 \,\leq \,r(qn)^3
 \left(\frac{ne}{qn}\right)^{3qn}\, \leq\, 2^{\alpha n} 
\]
for $d$ sufficiently large, where
$\alpha=(1/d)^{1/(2r-1)}$. This collection $\mathcal{C}$ will satisfy the
conditions of the lemma.

Fix an independent set $I$. For a subset $A \subset V$ with $I\subset A$,
and for $0 \leq j \leq r$, we define the set of edges 
$$
E_j(A)=\{e\in E: e\subset A, \, |e\cap I|\ge j\}\,.
$$
Let $P(j)$ be the statement
\[
 \mbox{for all } A \subset V \mbox{ with $I\subset A$ and }  \mu(A) \geq 1
 - 1/2r +  j/2r^2,\, |E_j(A)|\ge n d u^j/2r \mbox{ holds}\,.
\]
Statement $P(0)$ is true by~(\ref{eqn:mubig}), since
$|E_0(A)|=e(G[A])$.
Statement $P(r)$ is false, because $I$ is independent and so
$E_r(A)=\emptyset$.  There must therefore exist $j \in \{0,1,\ldots,r-1\}$
such that $P(j)$ is true and $P(j+1)$ is false. Fix a set $A$ witnessing
the falsity of $P(j+1)$; thus $I\subset A$, $\mu(A) \geq 1 - 1/2r +
(j+1)/2r^2$ and $|E_{j+1}(A)|<ndu^{j+1}/2r$.  For $v \in A$, let
\[
F_j(v) =
 \{e \in E : v\in e, \, e \in E_j(A),\, e\notin E_{j+1}(A) \}
= \{e: v\in e\subset A,\, |e\cap I|=j\}\,.
\]
Let $D=\{v\in A\setminus I: |F_j(v)|\ge du^j(1-u)/2r\}$. Note that
$I\subset A\setminus D$.

Consider an edge $e\in E_j(A\setminus D)$ with $e\notin E_{j+1}(A)$. Then
$e\subset A\setminus D$ and $|e\cap I|=j$. Since $j<r$ we can pick $v\in e$
with $v\notin I$. Now $E_j(A\setminus D)\subset E_j(A)$ so, by definition
of $F_j(v)$, we have $e\in F_j(v)$. Moreover, since $v\notin I$, the
definition of $D$ and the fact that $v\notin D$ imply $|F_j(v)|<
du^j(1-u)/2r$. Therefore, the total number of edges in $E_j(A\setminus D)$
but not in $E_{j+1}(A)$ is less than $|A\setminus D|du^j(1-u)/2r \le
ndu^j(1-u)/2r$. By the choice of $A$ as witness set, we know that
$|E_{j+1}(A)|<ndu^{j+1}/2r$ and so $E_j(A\setminus D)< 
ndu^j(1-u)/2r+ndu^{j+1}/2r=ndu^j/2r$. Since $P(j)$ is true, this means
$\mu(A\setminus D) < 1 - 1/2r + j/2r^2$. But $\mu(A) \geq 1 - 1/2r +
(j+1)/2r^2$ and therefore $\mu(D)>1/2r^2$.

Let $p=(6r/du^j)^{1/(r-1)}$, so $p^{r-1}du^j=6r$. Since $j\le r-1$, we
observe that $$
p\,\le\,\left( 6r\over d\right)^{1/(r-1)}{1\over u}
\,=\,\sqrt{3r}\left( 6r\over
  d\right)^{1/2(r-1)}\,=\,3ru\,=\frac{q}{5}\,.
$$ 
Let $R \subset I$ and $S \subset A\setminus I$ be random sets where each
vertex (of $I$ and $A\setminus I$ respectively) is included independently with
probability $p$.  By Markov's inequality, the inequalities $|R| \leq 5pn
\leq qn$ and $|S| \leq 5pn \leq qn$ each hold with probability at least
$4/5$.  Let $T = \Gamma_j(R,S) \cap I$. Then clearly, $I \subset
C_j(R,S,T)$ provided $\mu(\Gamma_j(R,S)\setminus T)\ge 1/4r^2$. So to
complete the proof, it is enough to show that the inequalities $|T| \leq
qn$ and $\mu(\Gamma_j(R,S)\setminus T)\ge 1/4r^2$ each hold with
probability at least~$4/5$, because then, with positive probability, all
four inequalities $|R|,|S|,|T|\le qn$ and $\mu(\Gamma_j(R,S)\setminus T)\ge
1/4r^2$ will hold.

A vertex $v \in I$ will be included in $\Gamma_j(R,S)$ (i.e., in $T$) if it
lies in an edge $e$ with $e = \{v\} \cup f \cup g$, $f\in R^{(j)}$, $g\in
S^{(r-j-1)}$. Therefore $e\subset A$ and $|e\cap I|= j+1$, which means
$e\in E_{j+1}(A)$. We know $|E_{j+1}(A)|<ndu^{j+1}/2r$. For an edge $e\in
E_{j+1}(A)$ with $|e\cap I|= j+1$, there are $j+1$ partitions of $e$ of the
form $e=\{v\} \cup f \cup g$ with $v\in I$, $f\in I^{(j)}$ and $g\in
(A\setminus I)^{(r-j-1)}$. For each such
partition, the probability that both $f\in R^{(j)}$ and $g\in S^{(r-j-1)}$
is $p^{r-1}$. So the expected size of $T$ is at most
\[
 rp^{r-1}ndu^{j+1}/2r\,=\,3run\,=\, qn/5.
\]
Applying Markov's inequality again implies that $|T| \leq qn$ with
probability at least $4/5$.

Recall that $D\cap I=\emptyset$ by definition of $D$, and so in particular
$D\cap T=\emptyset$. Let $D^*= D\setminus\Gamma_j(R,S)$. Then $D\setminus
D^*\subset\Gamma_j(R,S)\setminus T$, and so $\mu(\Gamma_j(R,S)\setminus T)
\ge \mu(D\setminus D^*)=\mu(D)-\mu(D^*)$. Let $v \in D$. Then $|F_j(v)|\ge
du^j(1-u)/2r>2du^j/5r$ (since $u$ is small). Each $e\in F_j(v)$ has a
partition $e = f \cup g$ with $f \in I^{(j)}$ and $g \in (A\setminus
I)^{(r-j)}$, where $v\in g$ because $v\notin I$. The probability that $f
\subset R$ and $g-\{v\} \subset S$ is $p^{r-1}$ and, because $G$ is simple,
these events over all~$e\in F_j(v)$ are independent. Hence the probability
that $v\in D^*$, that is, $v \not\in\Gamma_j(R,S)$, is at most
\[
 (1-p^{r-1})^{|F_j(v)|} \leq \exp\{-2p^{r-1}du^j/5r\} = \exp\{-12/5\} < 1/10.
\]
Now $\mu(D^*)=(1/nd)\sum_{v\in D^*} d(v)=(1/nd)\sum_{v\in D} d(v)I_v$ where
$I_v$ is the indicator of the event $v\in D^*$. Taking expectations, 
${\mathbb E}\mu(D^*)=(1/nd)\sum_{v\in D} d(v){\mathbb E}(I_v) < \mu(D)/10$, since 
${\mathbb E}(I_v)=\Pr(v\in D^*)<1/10$.
Markov's inequality implies that, with probability at least~$4/5$, 
$\mu(D^*)\le\mu(D)/2$ holds, and hence $\mu(\Gamma_j(R,S)\setminus T)
\ge \mu(D)-\mu(D^*)\ge\mu(D)/2>1/4r^2$. This completes the proof.
\end{proof}

We remark that the proof shows the theorem to be true for a smaller value
of~$\alpha$, namely $c(r)(\log d)/d^{1/2(r-1)}$ for some function $c(r)$
of~$r$, but the results of~\cite{ST2} are better still, with $\alpha =
c(r)(\log d)/d^{1/(r-1)}$, so we keep the present value for simplicity. One
might wonder why the bound here on $|\cal C|$ is worse than the bound
in~\cite{ST2}. It is not because of the random choice of $R$ and~$S$,
because in the context of the present algorithm random choice is quite
efficient, and a deterministic choice is unlikely to be much better. The
reason that the present method is relatively inefficient is that it uses
edges with exactly $j$ vertices in~$I$ for one value of $j$ only, and
ignores all other edges. The methods of~\cite{ST2} and~\cite{BMS}, which
are unrelated to the present method, are not lengthy to describe but are
nonetheless crafted carefully to use all edges and to be as efficient as
possible.

Applying Theorem~\ref{thm:simplecont} repeatedly, as described earlier,
we obtain the main theorem.

\begin{proof}[Proof of Theorem~\ref{thm:sparsecont}.]
  As we remarked earlier, let us apply Theorem~\ref{thm:simplecont} to $G$
  itself, and then again to each container so obtained, then to each of the
  new containers, and so on for each container with at least
  $\delta e(G)$ edges, until we obtain a collection $\cal C$ of
  containers $C$ with $e(G[C])< \delta e(G)$. Since,
  by~(\ref{eqn:musmall}), each application of
  Theorem~\ref{thm:simplecont} decreases the fraction of edges by $1-1/4r^2$,
  $\cal C$ is obtained after at most
  $k = \lceil (\log \delta)/\log(1-1/4r^2) \rceil +1$ levels of
  iteration, and so $|{\cal C}|\le 2^{k\alpha|G|}$, where $\alpha$ is the
  maximum over all applications of Theorem~\ref{thm:simplecont}. If
  $e(G[C]) \ge \delta e(G)$ then the average degree of~$G[C]$ is at least
  $\delta d$, and the result follows, provided, for the sake of a clean statement, the
  reader will indulge us by taking the value $\alpha(d)=c(r)(\log
  d)/d^{1/2(r-1)}$ rather than the weaker bound explicit in
  Theorem~\ref{thm:simplecont}.
\end{proof}

\section{Applications}

As remarked earlier, for regular hypergraphs, the condition $e(G[C]) <
\delta e(G)$ implies $|C|<(1-1/r+\delta/r)n$.  Plugging
this value into Theorem~2.1 of~\cite{ST1} immediately improves by a factor
of two the bound on the list colouring number in Theorem~1.1 of~\cite{ST1};
nevertheless the bound obtained remains a factor of two worse than the bound in
Theorem~1.3 of~\cite{ST2}, which is probably best possible.

We now give an application of Theorem~\ref{thm:sparsecont} in a situation
where the hypergraph of interest is not simple. In what follows, $H$ is a
fixed $\ell$-graph. We call another $\ell$-graph {\em $H$-free} if it has
no subgraph isomorphic to~$H$. The maximum size of an $H$-free $\ell$-graph
on $N$ vertices is denoted by ${\rm ex}(N,H)$, and $\pi(H)=\lim_{N\to\infty}{\rm
  ex}(N,H){N\choose\ell}^{-1}$ is the limiting maximum density of $H$-free
$\ell$-graphs.

\begin{theorem}\label{thm:ffree_cover}
Let $H$ be an $\ell$-graph and let $\epsilon>0$.
Then, if $N$ is large enough, there exists a collection
$\mathcal C$ of $\ell$-graphs on vertex set~$[N]$ such that
\begin{itemize}
 \item
   every $H$-free $\ell$-graph on vertex set $[N]$ is a subgraph
   of some $C\in\mathcal{C}$,
 \item
   every $C \in \mathcal{C}$ has at most $\epsilon N^{ v(H)}$
   copies of $H$,
   and $e(C) \le (\pi(H)+\epsilon) {N \choose \ell}$,
 \item
   $\log |\mathcal{C}| \leq N^{\ell-\sigma}$ where $\sigma=1/2e(H)$.
\end{itemize}
\end{theorem}

The meaning of this theorem is that every $H$-free $\ell$-graph is a
subgraph of one of just a few $\ell$-graphs that are nearly $H$-free. The
strength of the result can be measured by the bound on $\log
|\mathcal{C}|$. The bound $N^\ell$ is, of course, trivial, but any bound
where $\sigma$ is some positive constant is worthwhile. It is not
possible for $\sigma$ to exceed the value $m(H)=\max_{H' \subset H,\,e(H') >
  1} {(e(H')-1)/(v(H')-\ell)}$, and in fact a best possible bound was
obtained in~\cite[Theorem~1.3]{ST2}, but the method here is simpler. Any
value of $\sigma>0$, such as that given by Theorem~\ref{thm:ffree_cover},
immediately gives the following corollary, because each graph $C$ in the
statement of the theorem has at most $2^{(\pi(H)+\epsilon) {N \choose
    \ell}}$ subgraphs.

\begin{cor}\label{cor:ffree_count}
  Let $H$ be an $\ell$-graph. The number of $H$-free
  $\ell$-graphs on vertex set $[N]$ is $2^{(\pi(H)+o(1)) {N \choose \ell}}$.
\end{cor}

This corollary is the same as~\cite[Corollary~1.4]{ST2}. In the case
$\ell=2$, this corollary was proved for complete~$H$ by Erd\H{o}s, Kleitman
and Rothschild~\cite{EKR} and for general~$H$ by Erd\H{o}s, Frankl and
R\"odl~\cite{EFR}. Nagle, R\"odl and Schacht~\cite{NRS} proved it for
general~$\ell$ using hypergraph regularity methods. The present paper
offers the simplest known proof.

Theorem~\ref{thm:sparsecont} can, in a similar way, be used to give a
simple way to count the number of $\ell$-graphs which have no {\em induced}
copy of~$H$, and more generally to evaluate the probability that a random
$\ell$-uniform hypergraph $G^{(\ell)}(n,p)$ contains no induced copy
of~$H$. For $\ell=2$, the value when $p=1/2$ was determined by Pr\"omel and
Steger~\cite{PS} and for general $p$ by Bollob\'as and Thomason~\cite{BT}
(see also Marchant and Thomason~\cite{MT2}). For general~$\ell$ the value
for $p=1/2$ was given by Dotson and Nagle~\cite{DN}, again using hypergraph
regularity techniques.  We don't give details of the result, which is
identical to~\cite[Theorem~2.5]{ST2}. We merely point out that it can be
derived from a container theorem, as demonstrated in~\cite{ST2}, and that
the container theorem presented here can be used instead, via an argument
very similar to the one used to prove Theorem~\ref{thm:ffree_cover}.

Another application of Theorem~\ref{thm:ffree_cover} is the following
``sparse Tur\'an theorem''. Here the value of $\sigma$ does affect the
strength of the application.

\begin{cor}\label{thm:ffree_sparse}
Let $H$ be an $\ell$-graph and let $0<\gamma<1$.
For some $c>0$, for $N$ sufficiently large and for $p \ge cN^{-\sigma}$,
where $\sigma=1/2e(H)$, the following event holds with probability greater than 
$1-\exp\{-\gamma^3 p {N \choose \ell} / 512 \}$:
$$
\mbox{every $H$-free subgraph of $G^{(\ell)}(N,p)$ has at most 
$(\pi(H) + \gamma) p{N \choose \ell}$ edges.}
$$
\end{cor}

A stronger version of this corollary, with $\sigma=1/m(H)$, was conjectured
by Kohayakawa, \L{u}czak and R\"odl~\cite{KLR}; it was proved in the case
of strictly balanced~$H$ by Conlon and Gowers~\cite{CG} and in full
generality by Schacht~\cite{S}. The strong version follows easily
from~\cite[Theorem~1.3]{ST2}, as shown in~\cite{ST2}, and the same argument
gives Corollary~\ref{thm:ffree_sparse} from Theorem~\ref{thm:ffree_cover},
so we do not give details here. We remark that the point of the corollary
is how small the value of $p$ can be made: Szemer\'edi's regularity lemma
allows $p=o(1)$. We note that Kohayakawa, R\"odl and Schacht~\cite{KRS} and
Szab\'o and Vu~\cite{SV} both proved the corollary for complete 2-graphs
with $\sigma=1/(v(H)-1)$ (slightly better in the case of~\cite{SV}), but again
we believe the present proof is the shortest for some~$\sigma>0$. It yields
in a similar fashion weak versions of the other so-called K\L{R} conjectures.

The proof of Theorem~\ref{thm:ffree_cover} consists of finding a set of
containers for the independent sets in the hypergraph $G=G(N,H)$, which is
defined as follows. The $n={N\choose \ell}$ vertices of $G$ are the
$\ell$-sets in~$[N]$, that is, $V(G)=[N]^{(\ell)}$. The edges of $G$ are
the subsets of size $e(H)$ of $V(G)$ that form an $\ell$-graph isomorphic
to~$H$.

Given a subset $S\subset V(G)$, we can regard $S$ as the edges of an
$\ell$-graph with vertex set~$[N]$. The subset $S$ is independent in $G$ if
and only if $S$, regarded as an $\ell$-graph, is $H$-free. A set of
containers $\cal C$ for the independent sets in $G$ is thus a set of
$\ell$-graphs on vertex set $N$ such that every $H$-free graph is a subset
of one of these container graphs. Thus Theorem~\ref{thm:ffree_cover} is a
statement about the existence of a collection $\cal C$ of containers for
$G(N,H)$ having certain properties.

The stronger~\cite[Theorem~1.3]{ST2} was obtained by applying a container
result directly to $G(N,H)$. Here, we use the simpler
Theorem~\ref{thm:sparsecont} to give a set of containers with slightly
weaker properties. We cannot apply Theorem~\ref{thm:sparsecont} directly to
$G(N,H)$ because this hypergraph is not simple. We therefore apply it
instead to a subgraph $G_{\rm simple}(N,H)$ of $G(N,H)$. Each independent
set of $G(N,H)$ is independent in $G_{\rm simple}(N,H)$, so containers for
$G_{\rm simple}(N,H)$ will also be containers for $G(N,H)$. To show that
these containers have the properties claimed in
Theorem~\ref{thm:ffree_cover}, we need the following lemma.

\begin{lem}\label{lem:sparse}
  Let $\eta>0$ and $0<\rho<1$. Then, if $N$ is large enough, there exists a
  simple sub-hypergraph $G_{\rm simple}=G_{\rm simple}(N,H)$ of $G=G(N,H)$,
  such that $V(G_{\rm simple})=V(G)$ and $G_{\rm simple}$ has average
  degree at least $N^{\rho}$. Moreover, for all $S\subset V(G)$, if
  $e(G[S]) \ge \eta e(G)$, then $e(G_{\rm simple}[S]) \ge \eta
  e(G_{\rm simple})/2$.
\end{lem}

Given this lemma, the proof of Theorem~\ref{thm:ffree_cover} follows at
once, using the supersaturation theorem of Erd\H{o}s and
Simonovits~\cite{ES}, which itself has a very straightforward proof.

\begin{prop}[Erd\H{o}s and Simonovits~\cite{ES}]\label{prop:supersaturation}
Let $H$ be an $\ell$-graph and let $\epsilon > 0$. There exists $N_0$ and
$\eta > 0$ such that if $C$ is an $\ell$-graph on $N \ge N_0$ vertices
containing at most $\eta N^{v(H)}$ copies of $H$ then
$e(C) \le (\pi(H) + \epsilon){N \choose \ell}$.
\end{prop}
\begin{proof}[Proof of Theorem~\ref{thm:ffree_cover}.]
  Let $\epsilon>0$ be as given in the conditions of the theorem. Then let
  $\eta>0$ be given by Proposition~\ref{prop:supersaturation}. We may of
  course assume that $\eta\le\epsilon$.  Choose $\rho < 1$ so that
  $\rho/(2e(H)-1) >1/2e(H)$. Then apply Lemma~\ref{lem:sparse} to obtain
  $G_{\rm simple}$. Apply Theorem~\ref{thm:sparsecont} to $G_{\rm
    simple}$ with $\delta=\eta/2$, with $n={N\choose \ell}$ and $d \ge
  N^\rho$, so $d$ is large if $N$ is large, to obtain a collection $\cal
  C$ for the independent sets in $G_{\rm simple}$. As remarked before,
  every $H$-free $\ell$-graph $I$ on vertex set $[N]$ is an independent set
  in $G_{\rm simple}$ and is therefore contained in some subset $C\in
  {\cal C}$, which itself can be regarded as an $\ell$-graph on vertex
  set~$[N]$. We have $|{\cal C}|\le 2^{\beta n}$ where
  $\beta=(1/d)^{1/(2e(H)-1)}$.  Since $\rho/(2e(H)-1) >1/2e(H)$ we have
  $\log |{\cal C}|\le N^{\ell-1/2e(H)}$, as claimed.

  All that remains, then, is to verify the second assertion of the
  theorem. In the assertion, the number of copies of $H$ in $C$ is the same
  as $e(G[C])$. By Theorem~\ref{thm:sparsecont}, $e(G_{\rm simple}[C])
  < \delta e(G_{\rm simple})$. Since $\delta=\eta/2$,
  Lemma~\ref{lem:sparse} shows $e(G[C]) < \eta e(G) \le \epsilon e(G)$.  Now
  $e(G)$ is the number of copies of $H$ with vertices in~$[N]$ and so
  $e(G)<N^{v(H)}$.  So Proposition~\ref{prop:supersaturation} implies
  $e(C) \le (\pi(H) + \epsilon){N \choose \ell}$, completing the proof.
\end{proof}
\begin{proof}[Proof of Lemma~\ref{lem:sparse}.] 
  We form $G_{\rm simple}$ by randomly choosing edges of $G$ and then
  deleting a few so that the result is a simple hypergraph. Observe that
  ${N\choose v(H)}\le e(G)\le N^{v(H)}$, so that $e(G)=\Theta(N^h)$ where
  $h=v(H)$. We may assume that $H$ has more than one edge, and so
  $h\ge\ell+1\ge 3$. Call a pair $e,e'$ of edges of $G$ with
  $|e\cap e'|\ge 2$, an {\em overlapping pair}.  Notice that the number of
  overlapping pairs is the number of copies $H,H'$ of $H$ with vertices in
  $[N]$ that have at least two $\ell$-edges in common: $H$ and $H'$
  must share at least $\ell+1$ vertices and so the number of overlapping
  pairs is $O(N^{2h-\ell-1})$.

  Pick a number $\rho'$ with $\rho<\rho'<1$.  Let $G'$ be a subgraph of
  $G$ formed by picking edges independently and at random with
  probability $p=N^{-h+\ell+\rho'}$. Let $E$ be the number of edges
  of~$G'$.  We make use of standard bounds on the tail of the binomial
  distribution, to wit, if $X\sim{\rm Bi}(m,p)$ then $\Pr\{X\le (3/4){\mathbb
    E}X\}\le e^{-{\mathbb E}X/40}$, and the same bound holds for $\Pr\{X\ge
  (5/4){\mathbb E}X\}$ (see for example~\cite[Corollary 2.3]{JLR}). So if $A$ is
  the event $\{3{\mathbb E}E/4\le E\le 5{\mathbb E}E/4\}$ then $A$ holds with
  high probability, certainly more than $2/3$. Observe that if $A$ holds
  then $G'$ has $\Theta(N^{\ell+\rho'})$ edges.

  Let $F$ be the number of overlapping pairs in $G'$. Then ${\mathbb
    E}F=O(p^2N^{2h-\ell-1}) = O(N^{\ell+2\rho'-1})=o(N^{\ell+\rho'})$. Let
  $B$ be the event $\{F\le3{\mathbb E}F\}$. By Markov's inequality, $B$ holds
  with probability at least $2/3$.

  For each $S\subset V(G)$, let $C_S$ be the event that both $e(G[S]) \ge \eta
  e(G)$ and $e(G'[S]) \le 3pe(G[S])/4$ hold. Then the probability that
  $C_S$ occurs is at most $\exp(-pe(G[S])/40)=
  \exp(-p\Theta(N^h))=\exp(-\Theta(N^{\ell+\rho'}))$. Let $C$ be the event
  that $C_S$ does not hold for any~$S\subset V(G)$. There are
  $2^{|V(G)|}$ subsets~$S$, so the probability that $C$ fails to hold is at
  most $\exp(N^\ell-\Theta(N^{\ell+\rho'}))=o(1)$.

  There is therefore a positive probability that $A$, $B$ and $C$ all hold.
  Let $G'$ be a graph for which they all do hold, remove an edge from each
  overlapping pair, and call the result $G_{\rm simple}$. This graph
  has no overlapping pairs and so is simple. The number of edges is $E-F$.
  Since $A$ and $B$ hold, we have $E=\Theta(N^{\ell+\rho'})$ and
  $F=o(N^{\ell+\rho'})$, so $F=o(E)$ and $E-F = \Theta(N^{\ell+\rho'})
  >N^{\ell+\rho}$. The graph has fewer than $N^\ell$ vertices and so its
  average degree exceeds $N^\rho$. Finally, let $S\subset V(G)$ be such
  that $e(G[S]) \ge \eta e(G)$. The event $C$ holds, and so $C_S$ does not;
  thus $e(G'[S]) \ge 3pe(G[S])/4\ge 3p\eta e(G)/4 \ge (3/4)\eta(4/5) E$, the
  last inequality holding because $A$ holds. Therefore $e(G_{\rm
    simple}[S]) \ge 3\eta E/5-F$. But $F=o(E)$ so
  $e(G_{\rm simple}[S]) \ge 3\eta E/5-F\ge \eta(E-F)/2=\eta e(G_{\rm
    simple})/2$, which completes the proof.
\end{proof}

\end{document}